\documentclass[11pt]{article}

\usepackage{color,latexsym,amsthm,amsmath,amssymb,path,enumitem,url,bbm,subcaption,geometry,bbm}

\newcommand{\ignore}[1]{}

\newtheorem{theorem}{Theorem}[section]
\newtheorem{lemma}[theorem]{Lemma}

\newtheorem{problem}[theorem]{Problem}
\newcommand{\Proof}[1]
        {
        \noindent
        \emph{Proof #1.}~
        }
\newsavebox{\smallProofsym}                     
\savebox{\smallProofsym}                        %
        {
        \begin{picture}(7,7)                    %
        \put(0,0){\framebox(6,6){}}             %
        \put(0,2){\framebox(4,4){}}             %
        \end{picture}                           %
        }                                       %
\newcommand{\smalleop}[1]
        {
        \mbox{} \hfill #1~~\usebox{\smallProofsym}\!\!\!\!\!\!\
        }

\newcommand{\parag}[1]{\vspace{2mm}

\noindent{\bf #1} }

\usepackage{graphicx,psfrag}
\usepackage[rflt]{floatflt}

\usepackage{dsfont}
\newcommand{\NN}{\ensuremath{\mathbb N}}
\newcommand{\QQ}{\ensuremath{\mathbb Q}}
\newcommand{\ZZ}{\ensuremath{\mathbb Z}}
\newcommand{\RR}{\ensuremath{\mathbb R}}

\newcommand{\pts}{\mathcal P}

\newcommand{\lines}{\mathcal L}

\newcommand{\GCD}{\mathrm{gcd}}

\def\eps{{\varepsilon}}

\usepackage{mathtools} 
\DeclarePairedDelimiter{\floor}{\lfloor}{\rfloor}

\usepackage[colorlinks]{hyperref}

\begin{document}
\pagenumbering{arabic}

\title{Structural Szemer\'edi-Trotter for Lattices and their Generalizations\thanks{Part of the research work on this project was done as part of the 2021 NYC Discrete Math REU, funded by NSF award DMS-2051026.} }

\author{
Shival Dasu\thanks{Department of Mathematics, Indiana University, IN, USA. {\sl sdasu@iu.edu}}
\and
Adam Sheffer\thanks{Department of Mathematics, Baruch College, City University of New York, NY, USA.
{\sl adamsh@gmail.com}. Supported by NSF award DMS-1802059 and by PSCCUNY award 63580-00-51.}
\and
Junxuan Shen\thanks{California Institute of Technology, CA, USA, {\sl jshen@caltech.edu}. Supported by the Chung Ip Wing-Wah Memorial SURF Fellowship.}}

\date{}

\maketitle

\begin{abstract}
We completely characterize point--line configurations with $\Theta(n^{4/3})$ incidences when the point set is a section of the integer lattice. 
This can be seen as the main special case of the structural Szemer\'edi-Trotter problem. 
We also derive a partial characterization for several generalizations: (i) We rule out the concurrent lines case when the point set is a Cartesian product of an arithmetic progression and an arbitrary set. (ii) We study the case of a Cartesian product where one or both sets are generalized arithmetic progression. 
Our proofs rely on deriving properties of multiplicative energies.

\end{abstract}

\section{Introduction}

The Szemer\'edi--Trotter theorem \cite{SzemTrot83} is a central result in discrete geometry. 
It is an unusually helpful result, which is used in combinatorics, theoretical computer science, harmonic analysis, number theory, model theory, and more (for a few examples, see \cite{BomBour15,Bourgain05,CGS20,Demeter14}). 
Since this central result has been known for over 40 years, it is surprising that not much is known about the structural problem.
That is, not much is known about characterizing when this result is tight. 
A recent work of Silier and the second author introduced a new approach for addressing this structural problem \cite{SS21}.
In the current work, we further develop this approach, obtaining structural results in several special cases.

Throughout this paper, we work in $\RR^2$.
Consider a point set $\pts$ and a set of lines $\lines$.
A pair $(p,\ell)\in \pts\times \lines$ is an \emph{incidence} if the point $p$ is on the line $\ell$.
Let $I(\pts,\lines)$ be the number of incidences in $\pts\times\lines$.

\begin{theorem}[Szemer\'edi--Trotter] \label{th:SzemTrot}
Every finite point set $\pts$ and finite set of lines $\lines$ satisfy
\[ I(\pts,\lines) = O(|\pts|^{2/3}|\lines|^{2/3}+|\pts|+|\lines|). \]
\end{theorem}

In the bound of Theorem \ref{th:SzemTrot}, the cases where the linear terms dominate are considered as non-interesting.
For example, the term $|\pts|$ dominates when the number of points is larger than the square of the number of lines. 
The important term in the bound of Theorem \ref{th:SzemTrot} is $|\pts|^{2/3}|\lines|^{2/3}$.
The \emph{structural Szemer\'edi--Trotter problem} asks to characterize the sets of points and lines that form $\Theta(|\pts|^{2/3}|\lines|^{2/3})$ incidences.\footnote{Some authors refer to this as the the \emph{inverse} Szemer\'edi--Trotter problem.}
Even after decades of work, it is difficult to make conjectures about such sets. 
Below, we highlight some main open questions as \emph{problems}. 

When studying the structural Szemer\'edi--Trotter problem, it is common to assume that $|\pts|=|\lines|=n$, for simplicity.
In this case, we consider sets with $\Theta(n^{4/3})$ incidences. 
For $n\in \NN$, we set $[n] = \{1,2,\ldots,n\}$.
Erd\H os showed how to obtain $\Theta(n^{4/3})$ incidences when $\pts = [n^{1/2}]\times [n^{1/2}]$.
Decades later, Elekes \cite{Elekes01} showed how to obtain $\Theta(n^{4/3})$ incidences when $\pts = [n^{1/3}]\times [n^{2/3}]$.
Recently, it was shown in \cite{SS21} how to obtain $\Theta(n^{4/3})$ incidences when $\pts = [n^\alpha]\times [n^{1-\alpha}]$, for every $1/3\le \alpha \le 1/2$.
This infinite family of constructions includes the constructions of Erd\H os and Elekes.

When considering the above, one might conjecture that the point set must be a Cartesian product of arithmetic progressions. 
We can distort such a Cartesian product by applying projective transformations, replacing a fraction of the points, and applying other simple transformations. 
Recently, Guth and Silier \cite{GS21} introduced a configurations with $\Theta(n^{4/3})$ incidences where the point set is not a Cartesian product of arithmetic progressions, also after various simple transformation. 
Instead, the point set is a Cartesian product $A\times A$, where $A$ is a generalized arithmetic progression of dimension two (for a definition, see \eqref{eq:GAPdef} below). 
In a personal communication, Max Aires explained to us how to obtain constructions with generalized arithmetic progressions of any constant dimension. 

\begin{problem} \label{pr:GeneralizedLattice}
Consider a set $\pts$ of $n$ points and a set $\lines$ of $n$ lines, such that $I(\pts,\lines)=\Theta(n^{4/3})$. 
Must $\pts$ be a Cartesian product of constant-dimension generalized arithmetic progressions, possibly after a simple transformation?
\end{problem}

Silier and the second author \cite{SS21} proved the following structural result for Cartesian products.
Here $E^\times(S)$ stands for the multiplicative energy of $S$.
For a discussion of additive and multiplicative energies, see Section \ref{sec:MultEnergy}.

\begin{theorem} \label{th:StructuralSzemTrot} $\qquad$\\
(a) For $1/3<\alpha< 1/2$, let $A,B\subset \RR$ satisfy $|A|=n^\alpha$ and $|B|=n^{1-\alpha}$.
Let $\lines$ be a set of $n$ lines in $\RR^2$, such that $I(A\times B,\lines) = \Theta(n^{4/3})$.
Then at least one of the following holds:
\begin{itemize}
\item There exists $1-2\alpha\le \beta \le 2/3$ such that $\lines$ contains $\Omega(n^{1-\beta}/\log n)$ families of $\Theta(n^\beta)$ parallel lines, each with a different slope.
\item There exists $1-\alpha\le \gamma \le 2/3$ such that $\lines$ contains $\Omega(n^{1-\gamma}/\log n)$ disjoint families of $\Theta(n^\gamma)$ concurrent lines, each with a different center.
\end{itemize}
In either case, all the lines in these families are incident to $\Theta(n^{1/3})$ points of $A\times B$.\\[2mm]
(b) Assume that we are in the case of $\Omega(n^{1-\beta}/\log n)$ families of $\Theta(n^\beta)$ parallel lines. 
There exists $n^{2\beta} \le t \le n^{3\beta}$ such that, for $\Omega(n^{1-\beta}/\log^2 n)$ of these families, the additive energy of the $y$-intercepts is $\Theta(t)$.
Let $S$ be the set of slopes of these families. 
Then
\[ E^{\times} (S) \cdot t = \Omega(n^{3-\alpha}/ \log^{12}n). \]
\end{theorem}

In \cite{SS21}, part (a) of the theorem does not mention that all lines are incident to $\Theta(n^{1/3})$ points of $A\times B$.
This is immediate from the first line of the proof of that theorem and is important for the current work. 
With this property in mind, we say that a line is \emph{proper with respect to} $\pts$ if it is incident to $\Theta(n^{1/3})$ points of $\pts$.
When the point set is clear from the context, we say that the line is \emph{proper} and omit the point set.
Intuitively, part (b) of Theorem \ref{th:StructuralSzemTrot} states that, either the slopes have a high multiplicative energy or the $y$-intercepts have a high additive energy.\footnote{Intuitively and not rigorously, a set with a large additive energy behaves similarly to an arithmetic progression. 
A set with a large multiplicative energy behaves similarly to a geometric progression.}

Recently, Katz and Silier \cite{KS23} provided a very different approach for studying the structural Szemer\'edi-Trotter problem.
So far, the results that were produced by the two approaches do not overlap. 

\parag{Our results.}
We provide a stronger characterization of configurations with $\Theta(n^{4/3})$ incidences, in several of the main special cases. 
We say that a Cartesian product $A \times B$ is a \emph{lattice} if both $A$ and $B$ are arithmetic progressions. 
Until recently, all known configurations with $\Theta(n^{4/3})$ incidences were lattices. 
We completely characterize this case.
We define the set $\NN$ of natural numbers as not including zero.
For $i,j\in \NN$, let $\gcd(i,j)$ denote the greatest common divisor of $i$ and $j$.

\begin{theorem} \label{th:Lattice}
For a fixed $1/3<\alpha<1/2$, let $\lines$ be a set of $n$ lines such that $I([n^\alpha]\times [n^{1-\alpha}],\lines) = \Theta(n^{4/3})$.
Then 
\begin{itemize}
	\item Any concurrent family of proper lines in $\lines$ is of size $o(n^{1/3})$.\footnote{We recall that $A=o(B)$ means ``A is asymptotically smaller than $B$.'' The more common notation $A=O(B)$ means ``A is asymptotically smaller or equal to $B$.''}
    \item There exist $\Theta(n^{1/3})$ parallel families of $\Theta(n^{2/3})$ proper lines. 
    \item The slopes of the $\Theta(n^{1/3})$ parallel families are a constant portion of 
    \begin{align*}
    &\left\{\pm s/t\ :\ s,t\in \NN,\ \gcd(s,t)=1,\ t=\Theta(n^{\alpha-1/3}),\ s\le t\cdot n^{1-2\alpha}\right\}\bigcup\\[2mm]
	&\hspace{20mm}\left\{\pm s/t\ :\ s,t\in \NN,\ \gcd(s,t)=1,\ s=\Theta(n^{2/3-\alpha}),\ t\le s/ n^{1-2\alpha}\right\}.    
    \end{align*}
    \item Assume that a proper line contains at least $n^{1/3}/k$ points. Then the $y$-intercepts of a parallel family with slope $s/t$ form a constant portion of the set
    \[ \left\{j-i\cdot \frac{s}{t}\ :\  i\in [t],\ j\in [n^{1-\alpha}-sn^{1/3}/k],\ \text{ or }\ i\in [n^\alpha-tn^{1/3}/k],\ j\in [s]\right\}.\]    
\end{itemize}
\end{theorem}

Every lattice can be transformed to the form $[n^\alpha]\times [n^{1-\alpha}]$ with a translation and scaling along the axes.
Such transformations maintain lines and incidences.
When $\alpha<1/3$, the number of incidences is $o(n^{4/3})$.
Thus, in the case of a lattice, Theorem \ref{th:Lattice} implies that we cannot have even one large concurrent family of proper lines. 
We are not aware of any configuration that satisfies the concurrent case of Theorem \ref{th:StructuralSzemTrot}(a). 
It is tempting to conjecture that this case should be removed from Theorem \ref{th:StructuralSzemTrot}(a).

\begin{problem}
Can the case of many concurrent families be removed from the statement of Theorem \ref{th:StructuralSzemTrot}(a)?
\end{problem}

Theorem \ref{th:Lattice} also states that, in the case of a lattice, there are always $\Theta(n^{1/3})$ families of $\Theta(n^{2/3})$ parallel lines. 
In the context of Theorem \ref{th:StructuralSzemTrot}(a), we have that $\beta=2/3$ and that the logarithm can be removed from the bounds of the theorem.  
This is also the case in all of the constructions that we are aware of. 

\begin{problem}
In Theorem \ref{th:StructuralSzemTrot}(a), do we always have $\beta=2/3$? 
Can we always remove the logarithms from the bounds?
\end{problem}

We also study natural generalizations of the lattice case. 
First, we say that a Cartesian product $A \times B$ is a \emph{half-lattice} if at least one of $A$ and $B$ is an arithmetic progression. 
We prove the following result for half-lattices. 

\begin{theorem} \label{th:HalfLatticeConcurrentOnLattice}
Consider $\alpha>1/3$ and $B\subset \RR$ such that $|B|=n^{1-\alpha}$. 
Then the concurrent case of Theorem \ref{th:StructuralSzemTrot}(a) cannot occur with the half-lattice $[n^\alpha] \times B$.
\end{theorem}

Every half-lattice can be transformed to the form $[n^\alpha] \times B$ with a translation, scaling, and possibly switching the axes.
To prove Theorem \ref{th:HalfLatticeConcurrentOnLattice}, we require a basic property of multiplicative energy. 
Surprisingly, we could not find any results regarding such a property.
Posting this question online did not help either, so we ended up proving it on our own. 

\begin{theorem} \label{th:multEnergy}
Let $\alpha,\eps>0$. \\
(a) Consider $A \subset \RR$ such that $|A| = n$.
Then 
\[ E^{\times}(A, [n^\alpha]) =O(n^{1 + \alpha + \eps}). \]
(b) 
Consider $x\in \QQ$ and $A \subset \RR\setminus\{0\}$ such that $|A| = n^\alpha$.
Then
\[ E^{\times}(A, [n]+x) =O(n^{1+\alpha+\eps}+n^{2\alpha}). \]
(c) Consider $x\in \RR\setminus \QQ$ and $A \subset \RR\setminus\{0\}$ such that $|A| = n^\alpha$.
Then
\[ E^{\times}(A, [n]+x) =O(n^{1+\alpha}+n^{2\alpha}). \]
\end{theorem}

Our final result requires a few more definitions.
A \emph{generalized arithmetic progression} of dimension $d$ is defined as
\begin{equation} \label{eq:GAPdef} 
\left\{ a + \sum_{j=1}^{d} k_j b_j\ :\ k_j\in \NN \text{ and } 0\le k_j \le n_j-1 \text{ for every } 1\le j \le d \right\}, 
\end{equation}
where $a,b_1,\ldots,b_d \in \RR$ are fixed.
The \emph{size} of this a generalized arithmetic progression is $n_1\cdot n_2 \cdots n_d$.
Note that an arithmetic progression is a generalized arithmetic progression of dimension one. 
The set of $y$-intercepts in Theorem \ref{th:Lattice} is the union of two generalized arithmetic progression of dimension two. 

The \emph{sum set} of a set $A\subset \RR$ is
\[ A+A = \{a+a'\ :\ a,a'\in A \}.\]
As explained in more detail in Section \ref{sec:MultEnergy}, a set that satisfies $|A+A|=\Theta(|A|)$ is a constant portion of a constant-dimension generalized arithmetic progressions.
In all known point--line configurations with $\Theta(n^{4/3})$ incidences, the point set is a Cartesian product of generalized arithmetic progressions of a constant dimension. 
Equivalently, the point set is $A\times B$ where $|A+A|=\Theta(|A|)$ and $|B+B|=\Theta(|B|)$. 

We say that a Cartesian product $A\times B$ is a \emph{generalized lattice} if both $A$ and $B$ are constant-dimension generalized arithmetic progressions. 
We say that a Cartesian product $A\times B$ is a \emph{generalized half-lattice} if one of $A$ and $B$ is a constant-dimension generalized arithmetic progression.
We now state our results for generalized lattices and generalized half-lattices.

\begin{theorem} \label{th:ConcurrentSmallSumSet}
For $1/3< \alpha < 1/2$, let $A,B\subset \RR$ satisfy $|A|=n^{\alpha}$ and $|B|=n^{1-\alpha}$. \\[2mm]
(a) If $A\times B$ is a generalized lattice, then every concurrent family is of size $O(n^{1/3}\log n)$. \\[2mm]
(b) If $|B+B|=O(|B|)$ then every concurrent family is of size $O(n^{1/3+\alpha/2}\log^{1/2} n)$. \\[2mm]
(c) If $|A+A|=O(|A|)$ then every concurrent family is of size $O(n^{5/6-\alpha/2}\log^{1/2} n)$.
\end{theorem}

Theorem \ref{th:ConcurrentSmallSumSet} states that large concurrent families cannot occur in more main cases. 
For example, part (a) implies that the concurrent case of Theorem \ref{th:StructuralSzemTrot}(a) cannot occur with generalized lattices.
Part (b) of Theorem \ref{th:ConcurrentSmallSumSet} implies that the concurrent case cannot occur when $|B+B|=O(|B|)$ and $\alpha<4/9$.

In Section \ref{sec:MultEnergy}, we study additive and multiplicative energies and prove Theorem \ref{th:multEnergy}.
In Section \ref{sec:Lattices}, we study lattices and prove Theorem \ref{th:Lattice}.
Finally, in Section \ref{sec:generalizations} we study our lattice generalizations and prove Theorem \ref{th:HalfLatticeConcurrentOnLattice} and Theorem \ref{th:ConcurrentSmallSumSet}.

\section{Energies} \label{sec:MultEnergy}

In this section, we define additive and multiplicative energies, and then prove Theorem \ref{th:multEnergy}.
The \emph{additive energy} of a finite set $A\subset \RR$ is 
\[ E^+(A) = |\{(a,b,c,d)\in A^4\ :\ a+b=c+d\}|. \]
Additive energy is a central object in additive combinatorics. 
In some sense, it measures how $A$ behaves under addition. 
It is not difficult to show that $|A|^2<E^+(A)\le |A|^3$.
Intuitively, a large additive energy implies that there exists a large subset $A'\subset A$ that satisfies $|A'+A'|=O(|A'|)$.
For a more rigorous statement and many additional details about additive energy, see for example \cite{TaoVu06}.

Recall that generalized arithmetic progressions were defined in \eqref{eq:GAPdef}.
The following is a variant of Frieman's theorem over the reals.

\begin{theorem} \label{th:Freiman}
Let $A\subset \RR$ be a finite set with $|A+A| \le k |A|$ for some constant $k$.
Then $A$ is contained in a generalized arithmetic progression of size at most $c\cdot |A|$ and dimension at most $d$. 
Both $c$ and $d$ depend on $k$ but not on $|A|$.
\end{theorem}

Intuitively, Theorem \ref{th:Freiman} implies that a large $E^+(A)$ implies that a large subset of $A$ is a constant portion of a constant-dimension generalized arithmetic progression.
Alternatively, after removing some noise from it, $A$ is a constant portion of a constant-dimension generalized arithmetic progression. 

A \emph{generalized geometric progression} of dimension $d$ is defined as
\begin{equation*} 
\left\{ a \cdot \prod_{j=1}^{d} b_j^{k_j}\ :\ k_j\in \NN \text{ and } 0\le k_j \le n_j-1 \text{ for every } 1\le j \le d \right\}, 
\end{equation*}
where $a,b_1,\ldots,b_d \in \RR$ are fixed.
The \emph{multiplicative energy} of a finite set $A\subset \RR$ is 
\[ E^\times(A) = |\{(a,b,c,d)\in A^4\ :\ a\cdot b=c\cdot d\}|. \]
Similarly to the above, a large $E^\times(A)$ implies that a large subset of $A$ is a constant portion of a constant-dimension generalized geometric progression.

The following result is by Solymosi \cite{Solymosi09}.

\begin{lemma} \label{le:Solymosi}
Every finite $A\subset \RR$ satisfies $E^\times(A)=O(|A+A|^2 \log |A|)$.
\end{lemma}

As a first step towards proving Theorem \ref{th:multEnergy}(a), we show that it suffices to consider sets of nonzero integers. 
For $A\subset \RR$ and $x \in \RR$, we set $xA = \{x\cdot a\ :\ a \in A \}$. 

\begin{lemma} \label{le:RealsToInts}
Consider a finite $A \subset \RR\setminus\{0\}$ and $n \in \NN$. 
Then there exists $B \subset \ZZ\setminus\{0\}$ such that $|B| = |A|$ and
\[ E^{\times}(A, [n]) \leq E^{\times}(B, [n]). \]
\end{lemma}
\begin{proof}
We set $m = \max_{a\in A} |a|$. 
We note that multiplying all elements of $A$ by the same nonzero number does not affect $E^{\times}(A, [n])$.
Indeed, this does not change whether $a,b,c,d\in A\times[n]\times A\times[n]$ satisfy $a\cdot b=c\cdot d$.
Thus, by multiplying all elements of $A$ by a sufficiently large number, we may assume that $m\ge 10$.

For $a,b \in \RR$, we write $a \sim b$ if there exist $s,t \in \NN$ such that $\frac{s}{t} \cdot a = b$. 
We claim that, if a quadruple $(a_1,a_2,b_1,b_2)\in A^2 \times [n]^2$ contributes to $E^{\times}(A,[n])$, then $a_1\sim a_2$.
Indeed, by definition such a quadruple satisfies that $a_1\cdot b_1 = a_2 \cdot b_2$ and $b_1,b_2\in \NN$.

It is not difficult to verify that $\sim$ forms an equivalence relation. 
We denote the equivalence classes of the elements of $A$ under $\sim$ as $A_1,\ldots,A_k$.
In particular, these $k$ sets are disjoint and $A = \bigcup_{j=1}^k A_j$.
For each $j\in [k]$, we arbitrarily fix $a_j\in A_j$.
Combining the above leads to 
\[ E^{\times}(A, [n]) = \sum_{j=1}^k E^{\times}(A_j, [n]) = \sum_{j=1}^k E^{\times}\left(\frac{1}{a_j}A_j, [n]\right) \le E^{\times}\left(\bigcup_{j=1}^k \left(\frac{m^{10(j-1)}}{a_j}A_j\right), [n]\right). \]

We set $C = \bigcup_{j=1}^k \left(\frac{m^{10(j-1)}}{a_j}A_j\right)$, so the above states that $E^{\times}(A, [n]) \leq E^{\times}(C, [n])$. We note that $|A|=|C|$ and $C\subset \QQ\setminus\{0\}$.
Let $d$ be the common denominator of all elements of $C$.
We set $B = dC$, to obtain that $E^{\times}(A, [n]) \leq E^{\times}(B, [n])$, that $|A|=|B|$, and that $B\subset \ZZ\setminus\{0\}$.
\end{proof}

For $m \in \NN$ and $k\in \QQ$, we set 
\begin{align} 
r_{[m]/[m]}(k) &= \left|\left\{(x,y) \in [m]^2\ :\ \frac{x}{y} = k \right\}\right| \nonumber \\[2mm]
r_{[m]/[m]}^{(2)}\left(k\right) &=  \left|\left\{(x,y,z,w) \in [m]^4\ :\ \frac{x \cdot y}{z \cdot w} = k \right\}\right|. \label{eq:rm4def}
\end{align}

The following lemma states that $r_{[m]/[m]}^{(2)}$ cannot be large.

\begin{lemma} \label{le:fracrepsize}
Let $s,t,k\in \NN$ satisfy $\GCD(s,t)=1$ and $\max(s, t) =k$.
Then, for every $\eps>0$, we have that
\[  r_{[m]/[m]}^{(2)}\left(\frac{s}{t}\right) = O\left(\frac{m^{2 + \eps}}{k}\right). \]
\end{lemma}

In the bound of Lemma \ref{le:fracrepsize}, the parameter $\eps$ is fixed while $m$ is asymptotically large.

\begin{proof}[Proof of Lemma \ref{le:fracrepsize}.]
Without loss of generality, we assume that $s \le t$. 
For every $x,y,z,w\in [m]$ that satisfy $\frac{x\cdot y}{z\cdot w} = \frac{s}{t}$, there exist $s_1,s_2,p,q,t_1,t_2,p',q'\in [m]$ such that 
\begin{align*}
x &= s_1 \cdot p, \qquad y= s_2 \cdot q, \qquad s_1 \cdot s_2 = s,& \qquad & \qquad \\[2mm]
z &= t_1 \cdot p', \qquad w= t_2 \cdot q', \qquad t_1 \cdot t_2 = t,& \hspace{-10mm} \text{ and } \quad p\cdot q = p'\cdot q'.&  
\end{align*}

There exists $c\in \RR$ that satisfies the following for every $n\in \NN$:
The number of divisors of $n$ is at most $n^{c/\log\log n}$ (see for example \cite[Section 1.6]{Tao09}). 
Since $s\le x\cdot y$ and $x,y\le m$, we have that $s\le m^2$. 
The number of possible values for $s_1$ is at most the number of divisors of $s$, which is at most $m^{c/\log\log m} = O(m^{\eps/4})$.
After choosing $s_1$, the value of $s_2$ is fixed. 
Symmetrically, there are $O(m^{\eps/4})$ possible values for $t_1$ and then $t_2$ is fixed. 

We have that 
\[ p\cdot q = p'\cdot q' = z\cdot w/t = O(m^2/k).\]
Thus, there are $O(m^2/k)$ possible values for $p\cdot q$. 
For each value $p\cdot q$, the above bound on the number of divisors implies that there are $O(m^{\eps/4})$ possible values for $p$ and $O(m^{\eps/4})$ possible values for $p'$.
Combining the above leads to a total of $O(m^{2+\eps}/k)$ possible values for $s_1,s_2,p,q,t_1,t_2,p',q'$.
This in turn implies $O(m^{2+\eps}/k)$ possible values for $x,y,z,w$.
\end{proof}

For every positive rational number $k$, there exist unique $s,t \in \mathbb{N}$ that satisfy $\GCD(s,t)=1$ and $k=s/t$.
We denote these numbers as $s(k)$ and $t(k)$.
For every integer $1\le a\le n/\max(s(k),t(k))$, setting $x=a\cdot s(k)$ and $y=a\cdot t(k)$ gives that $x/y=k$. 
Thus,
\begin{equation} \label{eq:rnnEquivDef} 
r_{[m]/[m]}(k) = \floor[\Big]{\frac{m}{\max(s(k),t(k))}}.
\end{equation}

Consider a finite $A\subset \RR\setminus \{0\}$ and $m\in \NN$.
Let $(a_1,a_2,b_1,b_2)\in A^2\times [m]^2$ be a quadruple that contributes to $E^{\times}(A,[m])$.
Since $a_1 \cdot b_1 = a_2\cdot b_2$ is equivalent to $a_1/a_2 = b_1/b_2$, we have that
\begin{equation}\label{eq:EnergyAltForm}
    E^{\times}(A, [m]) = \sum_{(a,b) \in A^2} r_{[m]/[m]}\left(\frac{a}{b}\right).
\end{equation}

We now prove Theorem \ref{th:multEnergy} in the special case where $\alpha=1$.

\begin{lemma} \label{le:MultiplicativeN2}
Consider $A \subset \RR$ such that $|A| = n$, and let $\eps > 0$.
Then 
\[ E^{\times}(A, [n]) =O(n^{2 + \eps}). \]
\end{lemma}
\begin{proof}
If $0\in A$ then the number of quadruples of $A^2\times [n]^2$ that contribute to $E^{\times}(A, [n])$ and include 0 is $n^2$.
We may thus remove $0\in A$ without affecting the statement of the lemma. 
Then, by Lemma \ref{le:RealsToInts} we may assume that $A\subset \ZZ\setminus\{0\}$.

In each quadruple that contributes to $E^{\times}(A, [n])$, either both elements of $A$ are positive or both are negative.
If most of $E^{\times}(A, [n])$ comes from quadruples with negative elements of $A$, then we multiply all elements of $A$ by $-1$.
This does not change the value of $E^{\times}(A, [n])$.
Either way, we then remove from $A$ all of the negative elements.
This decreases $E^{\times}(A, [n])$ by at most half, which does not affect the asymptotic size of $E^{\times}(A, [n])$.
We now have that $A\subset \NN$ and that $|A|\le n$.

By \eqref{eq:EnergyAltForm}, the Cauchy--Schwartz inequality, and \eqref{eq:rm4def}, we obtain that
\begin{align}
   E^{\times}(A, [n]) &= \sum_{a_1 \in A}\sum_{a_2 \in A} r_{[n]/[n]}\left(\frac{a_1}{a_2}\right) \nonumber \\
   & \le \left(\sum_{a_1 \in A} 1\right)^{1/2}\cdot\left(\sum_{a_2 \in A}\left(\sum_{a_3 \in A} r_{[n]/[n]}\left(\frac{a_2}{a_3}\right)\right)^2\right)^{1/2} \nonumber \\
   &= n^{1/2}\cdot \left(\sum_{a_1 \in A} \sum_{(a_2, a_3) \in A^2} r_{[n]/[n]}\left(\frac{a_1}{a_2}\right) r_{[n]/[n]}\left(\frac{a_1}{a_3}\right)\right)^{1/2} \nonumber \\
   &= n^{1/2}\cdot \left(\sum_{(a_2, a_3) \in A^2} \sum_{a_1 \in A} r_{[n]/[n]}\left(\frac{a_1}{a_2}\right) r_{[n]/[n]}\left(\frac{a_3}{a_1}\right)\right)^{1/2} \nonumber \\
   & \le n^{1/2}\cdot\left( \sum_{(a_2, a_3) \in A^2}  r_{[n]/[n]}^{(2)}\left(\frac{a_3}{a_2}\right)\right)^{1/2}. \label{eq:IntermediateUpperE}
\end{align}
For the final transition above, recall that $r_{[n]/[n]}^{(2)}\left(\frac{a_3}{a_2}\right)$ is the number of solutions to $\frac{x\cdot y}{z\cdot w} = a_3/a_2$ with $x,y,z,w\in [n]$.
In the transition, we consider a subset of these solutions, where $x/z = a_1/a_2$ and $y/w = a_3/a_1$.

To have that $r_{[n]/[n]}^{(2)}(a_3/a_2)> 0$, we must have that $\max(s(a_3/a_2), t(a_3/a_2)) \le n^2$.
We may thus rewrite \eqref{eq:IntermediateUpperE} as
\[ E^{\times}(A, [n]) \le n^{1/2}\cdot \left( \sum_{j=0}^{2\log n} \sum_{(a_2, a_3) \in A^2 \atop 2^j\le \max(s(\frac{a_3}{a_2}), t(\frac{a_3}{a_2})) < 2^{j+1}} r_{[n]/[n]}^{(2)}\left(\frac{a_3}{a_2}\right)\right)^{1/2}. \]
Lemma \ref{le:fracrepsize} implies that 
\begin{equation} \label{eq:Alljs} 
E^{\times}(A, [n]) \le n^{1/2}\cdot \left( \sum_{j=0}^{2\log n} \sum_{(a_2, a_3) \in A^2 \atop 2^j\le \max(s(\frac{a_3}{a_2}), t(\frac{a_3}{a_2})) < 2^{j+1}} O\left(\frac{n^{2+\eps}}{2^j}\right)\right)^{1/2}. 
\end{equation}

\parag{Studying \eqref{eq:Alljs}.}
We consider the case where
\begin{equation} \label{eq:largej}
\sum_{j=0}^{\log n} \sum_{(a_2, a_3) \in A^2 \atop 2^j \le \max(s(\frac{a_3}{a_2}), t(\frac{a_3}{a_2})) < 2^{j+1}} O\left(\frac{n^{2+\eps}}{2^j}\right) \le \sum_{j=\log n}^{2\log n} \sum_{(a_2, a_3) \in A^2 \atop 2^j\le \max(s(\frac{a_3}{a_2}), t(\frac{a_3}{a_2})) < 2^{j+1}} O\left(\frac{n^{2+\eps}}{2^j}\right).
\end{equation}
That is, the case where large values of $j$ dominate the bound of \eqref{eq:Alljs}.
In this case, 
\begin{align*}
E^{\times}(A, [n]) &= n^{1/2}\cdot O\left( \sum_{j=\log n}^{2\log n} \sum_{(a_2, a_3) \in A^2 \atop 2^j\le \max(s(\frac{a_3}{a_2}), t(\frac{a_3}{a_2})) < 2^{j+1}} \frac{n^{2+\eps}}{2^j}\right)^{1/2} \\
&= n^{1/2}\cdot O\left( \sum_{j=\log n}^{2\log n} \sum_{(a_2, a_3) \in A^2 \atop 2^j\le \max(s(\frac{a_3}{a_2}), t(\frac{a_3}{a_2})) < 2^{j+1}} n^{1+\eps}\right)^{1/2} \\
&= n^{1/2}\cdot O\left(\sum_{(a_2, a_3) \in A^2} n^{1+\eps}\right)^{1/2} =n^{1/2}\cdot O\left(n^{3+\eps}\right)^{1/2} = O(n^{2+\eps}).
\end{align*}

It remains to consider the case where \eqref{eq:largej} is false. 
Applying \eqref{eq:EnergyAltForm} and then \eqref{eq:rnnEquivDef} leads to
\begin{align}
E^\times(A, [n]) &= \sum_{(a,b) \in A^2} r_{[n]/[n]}\left(\frac{a}{b}\right) =  \sum_{(a,b) \in A^2} \floor[\Big]{\frac{n}{\max(s(\frac{a}{b}), t(\frac{a}{b}))}} \nonumber \\[2mm]
            &=\Theta\left(\sum_{(a,b) \in A^2} \frac{n}{\max(s(\frac{a}{b}), t(\frac{a}{b}))}\right) =\Theta\left(\sum_{0\le j \le \log n} \sum_{(a,b) \in A^2 \atop 2^j\le \max(s(\frac{a}{b}),t(\frac{a}{b}))  < 2^{j+1} } \frac{n}{2^j} \right). \label{eq:SmalljApprox}
\end{align}
In the above, removing the floor function increases the amount by less than $n^2$.
This is allowed since, by definition, $E^\times(A, [n])\ge 2n^2-n$.

By \eqref{eq:Alljs}, recalling that we are in the case where \eqref{eq:largej} is false, and \eqref{eq:SmalljApprox}, we obtain that 
\begin{align*}
E^{\times}(A, [n]) &= n^{1/2}\cdot O\left( \sum_{j=0}^{\log n} \sum_{(a_2, a_3) \in A^2 \atop 2^j\le \max(s(\frac{a_3}{a_2}), t(\frac{a_3}{a_2})) < 2^{j+1}} \frac{n^{2+\eps}}{2^j}\right)^{1/2} \\[2mm]
&= n^{1+\eps/2}\cdot O\left( \sum_{j=0}^{\log n} \sum_{(a_2, a_3) \in A^2 \atop 2^j\le \max(s(\frac{a_3}{a_2}), t(\frac{a_3}{a_2})) < 2^{j+1}} \frac{n}{2^j}\right)^{1/2} = O\left( n^{1+\eps/2}\cdot E^\times(A,[n])^{1/2}\right). 
\end{align*}
Rearranging the above leads to $E^{\times}(A, [n])^{1/2}= O\left( n^{1+\eps/2}\right)$.
Squaring both sides completes the proof of this case.
\end{proof}

We are now ready to prove Theorem \ref{th:multEnergy}.
We first recall the statement of this result. 
\vspace{2mm}

\noindent {\bf Theorem \ref{th:multEnergy}.}
\emph{Let $\alpha,\eps>0$. \\
(a) Consider $A \subset \RR$ such that $|A| = n$.
Then 
\[ E^{\times}(A, [n^\alpha]) =O(n^{1 + \alpha + \eps}). \]
(b) 
Consider $x\in \QQ$ and $A \subset \RR\setminus\{0\}$ such that $|A| = n^\alpha$.
Then
\[ E^{\times}(A, [n]+x) =O(n^{1+\alpha+\eps}+n^{2\alpha}). \]
(c) Consider $x\in \RR\setminus \QQ$ and $A \subset \RR\setminus\{0\}$ such that $|A| = n^\alpha$.
Then
\[ E^{\times}(A, [n]+x) =O(n^{1+\alpha}+n^{2\alpha}). \] }
\begin{proof}
(a) By applying \eqref{eq:EnergyAltForm} and then \eqref{eq:rnnEquivDef}, we obtain that 
\begin{align*}
     n^{1-\alpha} \cdot E^\times&(A, [n^\alpha]) =  n^{1-\alpha}\cdot \sum_{(a,b) \in A^2} r_{[n^\alpha]/[n^\alpha]}\left(\frac{a}{b}\right) =
     \sum_{(a,b) \in A^2} n^{1-\alpha} \cdot \floor[\Big]{\frac{n^{\alpha}}{\text{max}(s(\frac{a}{b}), t(\frac{a}{b}))}}\\[2mm]
     & \leq \sum_{(a,b) \in A^2} \frac{n}{\text{max}(s(\frac{a}{b}), t(\frac{a}{b}))} \leq \sum_{(a,b) \in A^2} \floor[\Big]{\frac{n}{\text{max}(s(\frac{a}{b}), t(\frac{a}{b}))}} + |A|^2= E^\times(A,[n])+n^2.
\end{align*}
Rearranging and then applying Lemma \ref{le:MultiplicativeN2} leads to
\[ E^\times(A, [n^\alpha]) \le \frac{E^\times(A,[n])+n^2}{n^{1-\alpha}} = O(n^{1+\alpha+\eps}). \]

(b) Let $p,q\in \ZZ\setminus \{0\}$ satisfy $\gcd(p,q)=1$ and $x=p/q$.
We define 
\[ r_{p,q,n}(x) = \left|\left\{(b_1,b_2)\in [n]^2\ :\ x=(b_1+p/q)/(b_2+p/q) \right\}\right|. \]
Then, $E^\times([n]+p/q,A) = \sum_{a_1,a_2\in A}r_{p,q,n}(a_1/a_2)$.

We note that $r_{p,q,n}(x)>0$ only when $x$ is rational. 
Consider $s,t\in \NN$ and $b_1,b_2\in [n]$, so that $\gcd(s,t)=1$ and
\begin{equation} \label{eq:RationalRep}
\frac{s}{t} = \frac{b_1+p/q}{b_2+p/q}, \quad \text{ or equivalently } \ \frac{s}{t} = \frac{b_1q+p}{b_2q+p}. 
\end{equation}

We first assume that $s,t\neq 1$. 
By the above, there exists $k\in \NN$ such that $b_1q+p = ks$ and $b_2q+p=kt$.
This implies that $p\equiv ks \mod q$. 
Since $\gcd(p,q)=1$, we get that $\gcd(s,q)=1$, which in turn implies that $q$ has a multiplicative inverse modulo $s$. 
Then, $a_1q+p = ks$ leads to $a_1\equiv -pq^{-1} \mod s$.
In other words, there exists an integer $d$ such that $a_1 = (-pq^{-1}\mod s)+ds$.
Since $a_1\in [n]$, we conclude that the number of possible values for $a_1$ is at most $1+\lfloor n/s \rfloor$.
A symmetric argument shows that there are at most $1+\lfloor n/t \rfloor$ possible values for $a_2$.

By \eqref{eq:RationalRep}, fixing the value of one of $b_1,b_2$ determines the value of the other. 
Thus, the number of possible values for $(b_1,b_2)$ is at most $1+\lfloor n/\max\{s,t\} \rfloor$. 
Combining the above with \eqref{eq:rnnEquivDef}, \eqref{eq:EnergyAltForm}, and part (a) of the current theorem, leads to
\begin{align*}
E^\times([n]&+p/q,A) = \sum_{a_1,a_2\in A}r_{p,q,n}(a_1/a_2) = \sum_{a_1,a_2\in A \atop a_1/a_2 \in \QQ}r_{p,q,n}(a_1/a_2) \\ 
&\le \sum_{a_1,a_2\in A \atop a_1/a_2 \in \QQ}\left(1+\left\lfloor \frac{n}{\max\{s(a_1/a_2),t(a_1/a_2)\}} \right\rfloor\right)\\
&\le |A|^2 + \sum_{a_1,a_2\in A \atop a_1/a_2 \in \QQ}\left\lfloor \frac{n}{\max\{s(a_1/a_2),t(a_1/a_2)\}} \right\rfloor = |A|^2 + E(A,[n]) =O(n^{2\alpha}+n^{1+\alpha+\eps}).
\end{align*}

The above analysis ignores the cases where $s=1$ or $t=1$. 
We first consider the case where $s=1$ and $t>1$.
In this case, the above analysis still implies that the number of possible values for $(b_1,b_2)$ is at most $1+\lfloor n/t \rfloor$.
Thus, we still have that $r_{p,q,n}(a_1/a_2) \le 1+\left\lfloor \frac{n}{\max\{s(a_1/a_2),t(a_1/a_2)\}} \right\rfloor$, as required for the above analysis.
The case where $s>1$ and $t=1$ is symmetric.
Finally, when $s=t=1$, we immediately have that $r_{p,q,n}(a_1/a_2) \le 1+\left\lfloor \frac{n}{\max\{s(a_1/a_2),t(a_1/a_2)\}} \right\rfloor = n$, as required.
\vspace{2mm}

(c) We define
\[ r_{x,n}(y) = \left|\left\{p,q\in [n]\ :\ y=(p+x)/(q+x) \right\}\right|. \]
A quadruple $(a_1,q+x,a_2,p+x)\in (A\times ([n]+x))^2$ contributes to $E^{\times}(A, [n]+x)$ if and only if $a_1(p+x)=a_2(q+x)$.
Thus, for fixed $a_1,a_2\in A$, the number of quadruples with first element $a_1$ and third element $a_2$ is $r_{x,n}(a_2/a_1)$.
This implies that 
\[ E^{\times}(A, [n]+x) = \sum_{a_1,a_2\in A} r_{x,n}(a_2/a_1).\]

Consider $y\in \RR$ such that $r_{x,n}(y)\ge 2$.
Then there exist $p,p',q,q'\in [n]$ such that $(p,q)\neq(p',q')$ and 
\[ \frac{p+x}{q+x} = y = \frac{p'+x}{q'+x}. \]
Rearranging leads to 
\[ (p+x)(q'+x) = (p'+x)(q+x), \quad \text{or equivalently} \quad x(p+q') + pq' = x(p'+q)+p'q. \]

Since $x$ is irrational, we obtain that $p+q'=p'+q$ and $pq'=p'q$.
Rearranging the former leads to $p-p'=q-q'$, which is nonzero since $(p,q)\neq (p',q')$.
From $pq'=p'q$, we obtain that $pq'-pq =p'q-pq$, or equivalently $p(q'-q)=q(p'-p)$.
Since $p-p'=q-q'\neq 0$, we have that $p=q$.
This in turn implies that $y=(p+x)/(p+x)=1$.
We conclude that every $y\neq 1$ satisfies $r_{x,n}(y)\le 1$. 

We define the indicator function
\[ \mathbbm{1}_x(y) = \begin{cases} 
1, \qquad \text{exist } p,q\in [n]\ \text{such that } y=(p+x)/(q+x)\  \text{ and } p\neq q, \\
0, \qquad \text{otherwise.} 
\end{cases}\]
By definition, we have that $r_{x,n}(1)=n$.
Combining the above leads to
\begin{align*}
E^\times(A,[n]+x) = \sum_{a_1,a_2\in A} r_{x,n}(a_2/a_1) = |A|\cdot n + \sum_{a_1,a_2\in A} \mathbbm{1}_x(a_2/a_1) < |A|\cdot n+|A|^2.
\end{align*}
\end{proof}

\section{Lattices} \label{sec:Lattices}

In this section, we carefully characterize the configurations with $\Theta(n^{4/3})$ incidences where the point set is a lattice.
That is we prove Theorem \ref{th:Lattice}. 
Our proof technique heavily relies on \emph{Euler's totient function}.
The totient function of $j\in \NN$ is
\[ \phi(j) = |\{ i\in [j]\ :\ \gcd(i,j)=1 \}|. \]

We also consider the following variant of the totient function.
For $m,n\in\NN$, we define 
\[ \phi_m (n) = |\{ a \in [m]\ :\ \gcd(a, n) = 1\}|. \]
For proofs and further discussion of the following results, see \cite[Section 3]{BST22}.
For $n\in \NN$, let $w(n)$ denote the number of \emph{prime} divisors of $n$.

\begin{lemma} \label{eq:TotientProps}
Let $m,n\in \NN$. \\[2mm]
(a) $\displaystyle \sum_{j=1}^n \phi(j) = \Theta(n^2)$. \\[2mm]
(b) $\displaystyle \sum_{j=1}^{n}\frac{\phi(j)}{j} = \Theta(n)$. \\[2mm]
(c) $\displaystyle \phi_{m\cdot n}(n) = m\cdot \phi(n)$. \\[2mm]
(d) $\displaystyle \phi_m(n) = \frac{m}{n}\cdot \phi(n) + O(2^{w(n)})$. \\[2mm]
(e) $\displaystyle \sum_{r=1}^n 2^{w(r)} = O(n\log\log n)$.
\end{lemma}

We now recall the statement of Theorem \ref{th:Lattice}.
\vspace{2mm}

\noindent {\bf Theorem \ref{th:Lattice}.}
\emph{For a fixed $1/3<\alpha<1/2$, let $\lines$ be a set of $n$ lines such that $I([n^\alpha]\times [n^{1-\alpha}],\lines) = \Theta(n^{4/3})$.
Then 
\begin{itemize}
	\item Any concurrent family of proper lines in $\lines$ is of size $o(n^{1/3})$.
    \item There exist $\Theta(n^{1/3})$ parallel families of $\Theta(n^{2/3})$ proper lines. 
    \item The slopes of the $\Theta(n^{1/3})$ parallel families are a constant portion of 
    \begin{align*}
    &\left\{\pm s/t\ :\ s,t\in \NN,\ \gcd(s,t)=1,\ t=\Theta(n^{\alpha-1/3}),\ s\le t\cdot n^{1-2\alpha}\right\}\bigcup\\[2mm]
	&\hspace{20mm}\left\{\pm s/t\ :\ s,t\in \NN,\ \gcd(s,t)=1,\ s=\Theta(n^{2/3-\alpha}),\ t\le s/ n^{1-2\alpha}\right\}.    
    \end{align*}
    \item Assume that a proper line contains at least $n^{1/3}/k$ points. Then the $y$-intercepts of a parallel family with slope $s/t$ form a constant portion of the set
    \[ \left\{j-i\cdot \frac{s}{t}\ :\  i\in [t],\ j\in [n^{1-\alpha}-sn^{1/3}/k],\ \text{ or }\ i\in [n^\alpha-tn^{1/3}/k],\ j\in [s]\right\}.\]    
\end{itemize} }

\begin{proof}
Every axis-parallel line is incident to either $n^{\alpha},n^{1-\alpha}$, or zero points of $[n^\alpha]\times [n^{1-\alpha}]$.
Thus, there are no proper axis-parallel lines.
In our analysis we focus on lines with positive slopes. 
The analysis of the lines with negative slopes is symmetric. 
We fix $k\in \NN$, such that a proper line is incident to more than $n^{1/3}/k$ points and to at most $kn^{1/3}$.

\begin{figure}[ht]
\centerline{\includegraphics[width=0.23\textwidth]{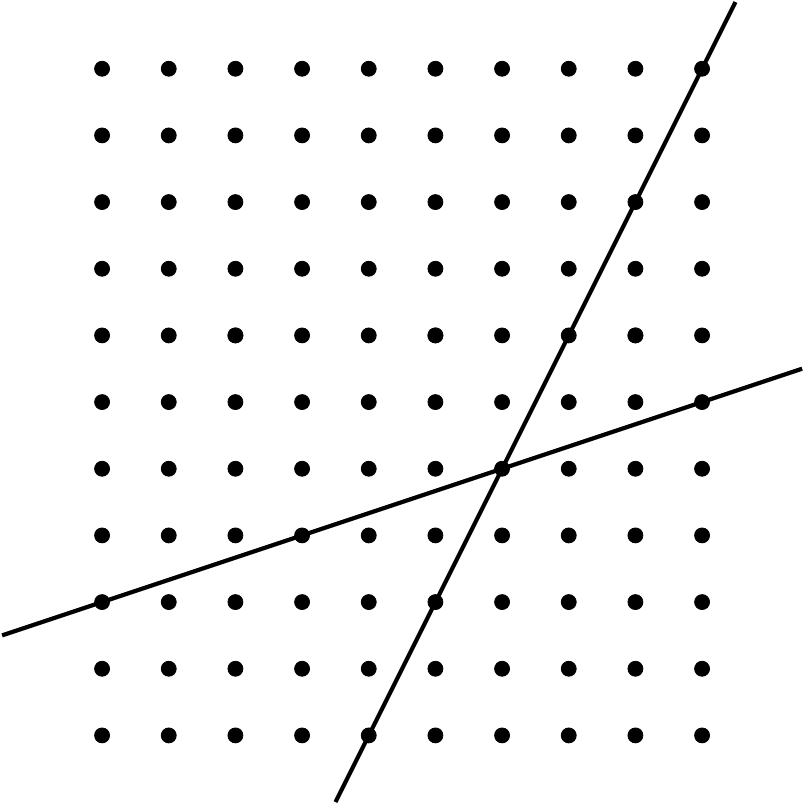}}
\caption{For a lattice of size $10\times 11$, a line is steep if its slope is larger than $11/10$. 
When $s=1$ and $t=3$, a line contains at most $\lceil 10/3\rceil=4$ points.
When $s=2$ and $t=1$, a line contains at most $\lceil 11/2\rceil=6$ points. }
\label{fi:slopeInc}
\end{figure}

\parag{Slopes with at least one proper line.}
The slope of a line $\ell$ that is incident to two points $(a,b),(a',b')\in [n^\alpha]\times [n^{1-\alpha}]$ is $(b-b')/(a-a')$.
This implies that all proper lines have rational slopes. 
Since we focus on positive slopes, there exist unique $s,t\in \NN$ such that the slope of $\ell$ is $s/t$ and $\GCD(s,t)=1$.
We note that $\ell$ contains a point of the lattice $\ZZ^2$ every $t$ columns and every $s$ rows.
This implies that the number of points of $[n^\alpha]\times [n^{1-\alpha}]$ that are on $\ell$ is at most $\min(\lceil n^\alpha/t\rceil,\lceil n^{1-\alpha}/s\rceil)$.
We say that $\ell$ is \emph{steep} if $s/t > n^{1-2\alpha}$.
If $\ell$ is steep then it contains at most $\lceil n^\alpha/t \rceil$ points of $[n^\alpha]\times [n^{1-\alpha}]$.
Otherwise, $\ell$ contains at most $\lceil n^{1-\alpha}/s\rceil$ points of $[n^\alpha]\times [n^{1-\alpha}]$.
See Figure \ref{fi:slopeInc}.

We first consider slopes of proper non-steep lines. 
By the preceding paragraph, a non-steep line with slope $s/t$ (and $\gcd(s,t)=1$) is incident to at most $\lceil n^\alpha/t\rceil$ points of $[n^\alpha]\times [n^{1-\alpha}]$.
We must thus have that $n^\alpha/t > n^{1/3}/k$, or equivalently $t<k\cdot n^{\alpha-1/3}$. 
Since we are in the case where $s\le t\cdot n^{1-2\alpha}$, after fixing $t$ the number of valid values for $s$ is $\phi_{t\cdot n^{1-2\alpha}}(t)$.
By Lemma \ref{eq:TotientProps}(c) and Lemma \ref{eq:TotientProps}(a), the number of non-steep slopes with at least one proper line is at most 
\[ \sum_{t=1}^{k\cdot n^{\alpha-1/3}} \phi_{t\cdot n^{1-2\alpha}}(t) = n^{1-2\alpha}\cdot \sum_{t=1}^{k\cdot n^{\alpha-1/3}} \phi(t) = n^{1-2\alpha}\cdot O(n^{2\alpha-2/3}) = O(n^{1/3}). \]

We next consider slopes of proper steep lines. 
By the above, a steep line with slope $s/t$ is incident to at most $n^{1-\alpha}/s$ points of $[n^\alpha]\times [n^{1-\alpha}]$.
We thus have that $\lceil n^{1-\alpha}/s\rceil > n^{1/3}/k$, or equivalently $s<k\cdot n^{2/3-\alpha}$. 
Since we are in the case where $s> t\cdot n^{1-2\alpha}$, after fixing $s$ the number of valid values for $t$ is $\phi_{s/ n^{1-2\alpha}}(s)$.
By Lemma \ref{eq:TotientProps}(d), Lemma \ref{eq:TotientProps}(e), and Lemma \ref{eq:TotientProps}(a), the number of steep slopes with at least one proper line is at most 
\begin{align*} 
\sum_{s=1}^{k\cdot n^{2/3-\alpha}} \phi_{s/ n^{1-2\alpha}}(s) &= \frac{1}{n^{1-2\alpha}}\sum_{s=1}^{k\cdot n^{2/3-\alpha}}\phi(s) + \sum_{s=1}^{k\cdot n^{2/3-\alpha}}O(2^{w(s)}) \\[2mm] 
&= O\left(\frac{n^{4/3-2\alpha}}{n^{1-2\alpha}}+n^{2/3-\alpha}\log \log n \right) = O(n^{1/3}). 
\end{align*}

Since the analysis of the negative slopes is symmetric, the number of slopes that have at least one proper line is $O(n^{1/3})$.
This implies that any concurrent family of proper lines is of size $O(n^{1/3})$.
That is, in the context of Theorem \ref{th:StructuralSzemTrot}, we are in the case of many families of parallel lines.

Our next goal is to study the proper lines of a fixed slope $s/t$.
In particular, we characterize the slopes that have many proper lines. 

\parag{Non-steep slopes with many proper lines.}
We fix $s,t\in \NN$ with $\gcd(s,t)=1$ and $s/t\le n^{1-2\alpha}$.
Recall that a line with slope $s/t$ is incident to at most $\lceil n^\alpha/t \rceil$ points of $[n^\alpha]\times [n^{1-\alpha}]$.
We also recall that, in this case $t\le k\cdot n^{\alpha-1/3}$.

\begin{figure}[ht]
    \centering
    \begin{subfigure}[b]{0.2\textwidth}
    \centering
        \includegraphics[width=0.97\textwidth]{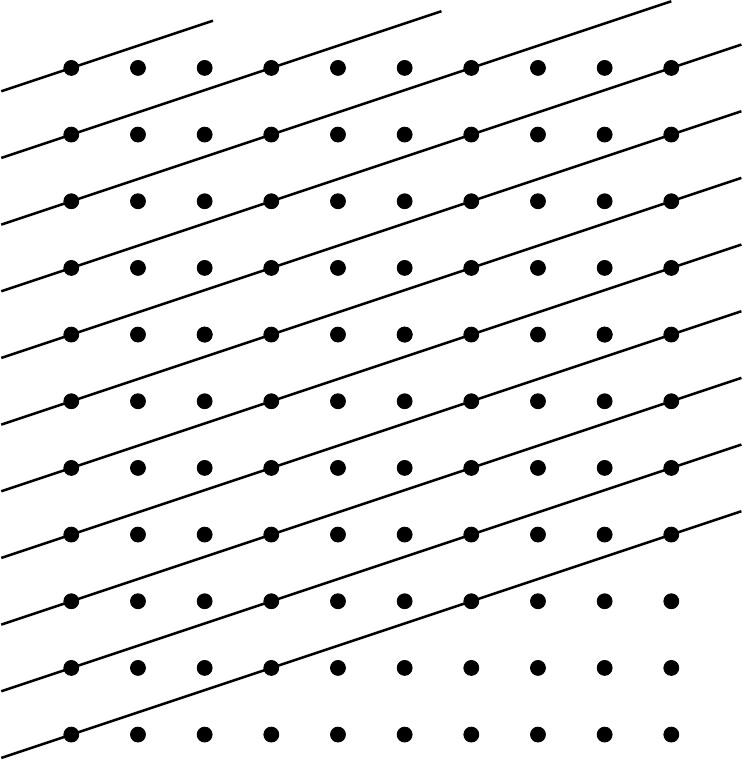}
        \caption{}
    \end{subfigure}
    \hspace{1cm}
    \begin{subfigure}[b]{0.18\textwidth}
        \centering
        \includegraphics[width=\textwidth]{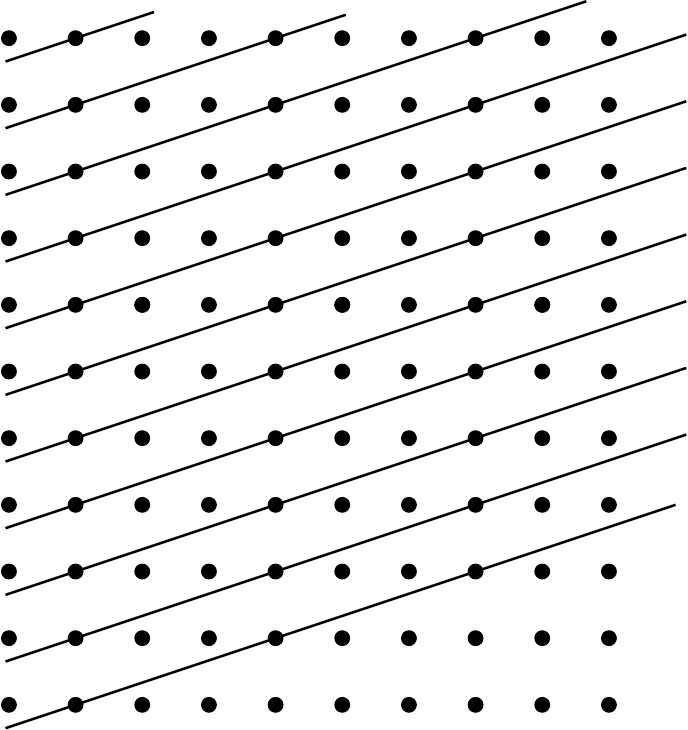}
        \caption{}
    \end{subfigure}
    \vspace{2mm}
    \caption{(a) The lines of slope $1/3$ that contain a point from the leftmost column of the lattice. (b) When moving from leftmost column to one of the next $t-1$ leftmost columns, some lines may be incident to one less point.}
    \label{fi:Leftmost}
\end{figure}

Let $C$ be the leftmost column of $[n^\alpha]\times [n^{1-\alpha}]$ and consider all lines with slope $s/t$ that are incident to a point of $C$.
See Figure \ref{fi:Leftmost}(a).
The lines that contain the $s$ top points of $C$ are incident to one point of $[n^\alpha]\times [n^{1-\alpha}]$.
The lines that contain the next $s$ highest points of $C$ are incident to two points of $[n^\alpha]\times [n^{1-\alpha}]$.
The lines that contain the next $s$ highest points are incident to three points, and so on. 
After $s$ lines that are incident to $\lceil n^\alpha/t\rceil-1$ points, the remaining $n^{1-\alpha}-s(\lceil n^\alpha/t\rceil-1)$ lines are incident to $\lceil n^\alpha/t\rceil$ points. 

A line contains a point of $\ZZ^2$ every $t$ columns. 
We may thus repeat the analysis of the preceding paragraph for each of the $t$ leftmost columns of $[n^\alpha]\times [n^{1-\alpha}]$ without considering the same line twice. 
The only difference is that some lines may be incident to one point less than their corresponding line from the analysis of $C$. 
See Figure \ref{fi:Leftmost}(b).

\begin{figure}[ht]
\centerline{\includegraphics[width=0.23\textwidth]{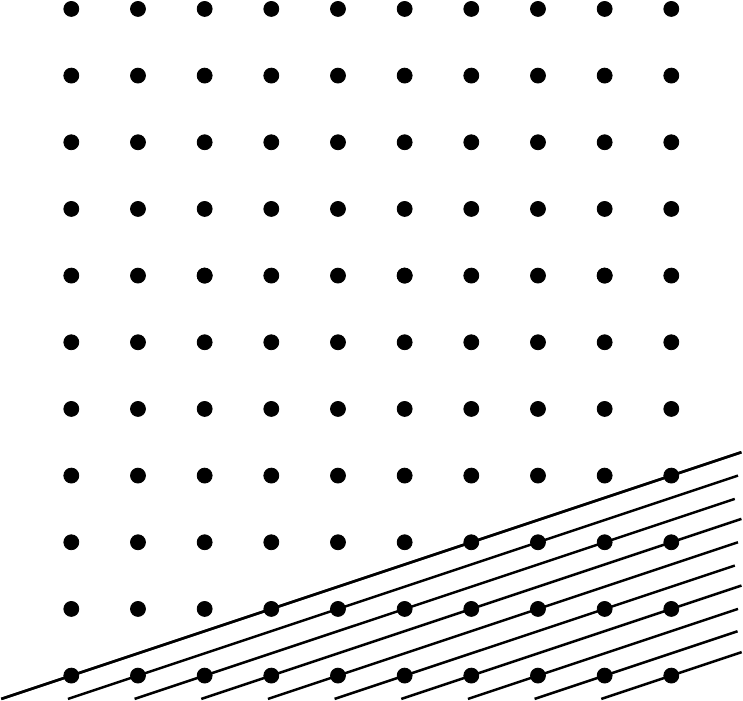}}
\caption{The lines with slope $1/3$ that contain a point from the bottom row of the lattice. }
\label{fi:Lowest}
\end{figure}

Let $R$ be the lowest row of $[n^\alpha]\times [n^{1-\alpha}]$ and consider all lines with slope $s/t$ that and are incident to a point of $R$.
See Figure \ref{fi:Lowest}.
The lines that contain the $t$ leftmost points of $R$ are incident to one point of $[n^\alpha]\times [n^{1-\alpha}]$.
The lines that contain the next $t$ leftmost points of $R$ are incident to two points of $[n^\alpha]\times [n^{1-\alpha}]$.
The lines that contain the next $t$ leftmost points are incident to three points, and so on. 
After $s$ lines that are incident to $\lceil n^\alpha/t\rceil-1$ points, at most $s$ lines are incident to $\lceil n^\alpha/t\rceil$ points. 
A line contains a point of $\ZZ^2$ every $s$ rows. 
We may thus repeat the above analysis for the $s$ bottom rows of $[n^\alpha]\times [n^{1-\alpha}]$ without considering the same line twice.
As before, some lines may be incident to one point less than their corresponding line from the analysis of $R$.

Every line with slope $s/t$ that is incident to at least one point of $[n^\alpha]\times [n^{1-\alpha}]$ contains a point from the bottom $s$ rows or from the $t$ leftmost columns.
Thus, it suffices to consider these rows and columns. 
We note that $s\cdot t$ lines are counted twice, once for the $s$ bottom rows and once for the $t$ leftmost columns. 

If $t =o(n^{\alpha-1/3})$ then $n^\alpha/t = \omega(n^{1/3})$. 
In this case, each of the $t$ leftmost columns contains $s(kn^{1/3}-n^{1/3}/k)$ points on proper lines with slope $s/t$. 
Similarly, each of the $s$ bottom rows contains $t(kn^{1/3}-n^{1/3}/k)$ points on such lines. 
In total, there are $\Theta(s\cdot t\cdot n^{1/3})$ proper lines with slope $s/t$. 
Since $t =o(n^{\alpha-1/3})$ and $s \le t\cdot n^{1-2\alpha}$, the number of such lines is
\[ \Theta(s\cdot t \cdot n^{1/3}) = O(t^2\cdot n^{1-2\alpha}\cdot n^{1/3}) = o(n^{2/3}). \]

By the above, when $t =o(n^{\alpha-1/3})$, every parallel family is of size $o(n^{2/3})$.
Since there are $O(n^{1/3})$ slopes with proper lines, such values of $t$ lead to $o(n^{4/3})$ incidences. 

We now consider the case where $t=\Omega(n^{\alpha-1/3})$.
Since $t=O(n^{\alpha-1/3})$, this is the case where $t=\Theta(n^{\alpha-1/3})$.
In this case, the lines that contain $\lceil n^\alpha/t\rceil$ points of $[n^\alpha]\times [n^{1-\alpha}]$ are proper.
Each of the $s$ bottom rows contains $n^\alpha-tn^{1/3}/k$ points that are incident to a proper line. 
Each of the leftmost $t$ columns contains $n^{1-\alpha}-sn^{1/3}/k$ points that are incident to such a line. 
By taking $k$ to be sufficiently large, since $t=\Theta(n^{\alpha-1/3})$, and since $s\le t\cdot n^{1-2\alpha}$, the number of proper lines with slope $s/t$ is  
\begin{align*} 
t\cdot(n^{1-\alpha}-sn^{1/3}/k) + s(n^\alpha-tn^{1/3}/k) = \Theta(tn^{1-\alpha} + sn^\alpha) = \Theta(n^{2/3}).
\end{align*}

We say that a slope is \emph{rich} if there are $\Theta(n^{2/3})$ proper lines with this slope. 
By the above, all slopes of the current case are rich. 
Since $t=\Theta(n^{\alpha-1/3})$, we may set $k_t\in \NN$ such that $n^{\alpha-1/3}/ k_t\le t\le n^{\alpha-1/3} \cdot k_t$.
After fixing a value for $t$, the number of valid values for $s$ is $\phi_{t\cdot n^{1-2\alpha}}(t)$.
By Lemma \ref{eq:TotientProps}(c) and Lemma \ref{eq:TotientProps}(a), the number of non-steep rich slopes is 
\begin{align*} 
\sum_{t=n^{\alpha-1/3}/k_t}^{n^{\alpha-1/3}\cdot k_t} \phi_{t\cdot n^{1-2\alpha}}(t) = \sum_{t=n^{\alpha-1/3}/k_t}^{n^{\alpha-1/3}\cdot k_t} &n^{1-2\alpha}\cdot \phi(t) = n^{1-2\alpha}\cdot \sum_{t=1}^{n^{\alpha-1/3}\cdot k_t} \phi(t) - n^{1-2\alpha}\cdot\sum_{t=1}^{n^{\alpha-1/3}/k_t} \phi(t)\\[2mm]
&=\Theta\left(n^{1-2\alpha}\left(n^{2\alpha-2/3}\cdot k_t^2 - n^{2\alpha-2/3}/ k_t^2\right)\right) = \Theta(n^{1/3}).
\end{align*}

We conclude that there are $\Theta(n^{1/3})$ families of $\Theta(n^{2/3})$ proper parallel lines with non-steep slopes. 
These lines lead to $\Theta(n^{4/3})$ incidences.

\parag{Steep slopes with many proper lines.}
This case can be studied in the same way as the case of non-steep slopes. 
Instead of repeating the same analysis, we focus on the changes required for this case.
We fix $s,t\in \NN$ such $\gcd(s,t)=1$ and $s/t> n^{1-2\alpha}$.
We recall that a line with slope $s/t$ is incident to at most $\lceil n^{1-\alpha}/s \rceil$ points of $[n^\alpha]\times [n^{1-\alpha}]$.
We also recall that a proper line with steep slope $s/t$ satisfies that $s\le k\cdot n^{2/3-\alpha}$.

\begin{figure}[ht]
\centerline{\includegraphics[width=0.2\textwidth]{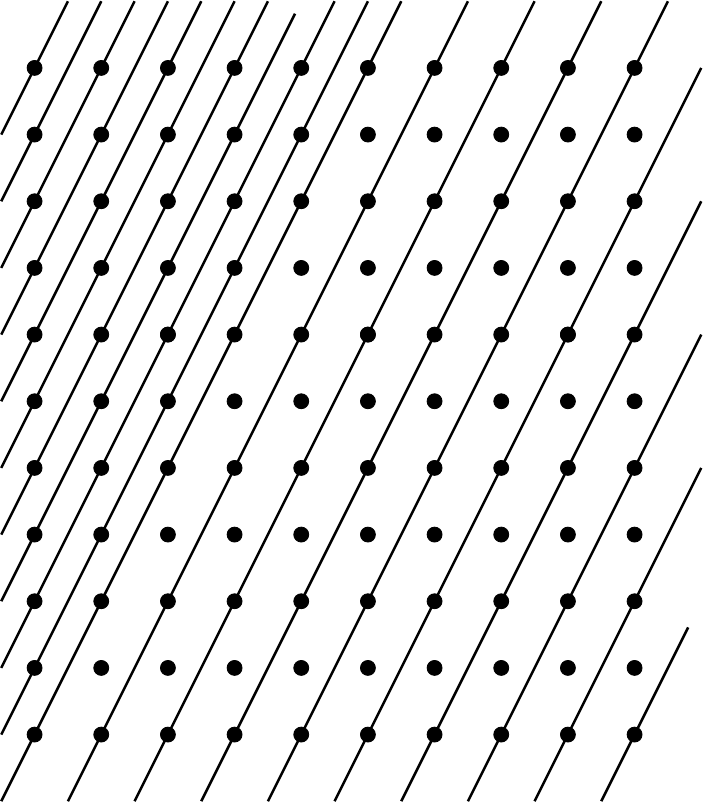}}
\caption{For steep lines, we expect more points on the bottom row to be incident to lines that contain $\lceil n^{1-\alpha}/s \rceil$ points, rather than points on the leftmost column. }
\label{fi:steep}
\end{figure}

The columns and rows $[n^\alpha]\times [n^{1-\alpha}]$ switch the roles that they had in the analysis of the case of non-steep slopes. 
On the leftmost column, at most $s$ points are on lines with slope $s/t$ that are incident to $\lceil n^{1-\alpha}/s \rceil$ points of $[n^\alpha]\times [n^{1-\alpha}]$.
On the bottom row, $n^\alpha - t(\lceil n^{1-\alpha}/s \rceil-1)$ points are on lines with slope $s/t$ incident to $\lceil n^{1-\alpha}/s \rceil$ points of $[n^\alpha]\times [n^{1-\alpha}]$. 
See Figure \ref{fi:steep}.

We rearrange $s/t> n^{1-2\alpha}$ as $t < s/n^{1-2\alpha}$.
When $s= o(n^{2/3-\alpha})$, the number of proper lines with slope $s/t$ is 
\[ \Theta(s\cdot t\cdot n^{1/3}) = O(s^2/n^{1-2\alpha}\cdot n^{1/3}) = o(n^{4/3-2\alpha}/n^{1-2\alpha}\cdot n^{1/3}) = o(n^{2/3}). \]

By the above, when $s= o(n^{2/3-\alpha})$, every parallel family is of size $o(n^{2/3})$.
Since there are $O(n^{1/3})$ slopes with proper lines, such values of $s$ lead to $o(n^{4/3})$ incidences. 

We now consider the case where $s= \Theta(n^{2/3-\alpha})$.
In this case, the lines that contain $\lceil n^{1-\alpha}/s\rceil$ points of $[n^\alpha]\times [n^{1-\alpha}]$ are proper.
Imitating the analysis of the non-steep case, each of the leftmost $t$ columns contains $n^{1-\alpha} - s\cdot n^{1/3}/k$ proper lines with slope $s/t$.
Each of the bottom $s$ rows contains $n^\alpha-tn^{1/3}/k$ points that are incident to such lines. 
By taking a sufficiently large $k$, since $s= \Theta(n^{2/3-\alpha})$, and since $t < s/n^{1-2\alpha}$, the number of proper lines with slope $s/t$ is 
\[ s\cdot(n^\alpha-tn^{1/3}/k) + t\cdot(n^{1-\alpha} - s\cdot n^{1/3}/k) = \Theta(sn^{\alpha} +t\cdot n^{1-\alpha})= \Theta(n^{2/3}). \] 

Since $s= \Theta(n^{2/3-\alpha})$, we may set $k_s\in \NN$ such that $n^{2/3-\alpha}/ k_s\le s\le n^{2/3-\alpha} \cdot k_s$.
After fixing a value for $s$, the number of valid values for $t$ is $\phi_{s/n^{1-2\alpha}}(s)$.
By Lemma \ref{eq:TotientProps}(d), Lemma \ref{eq:TotientProps}(e), and Lemma \ref{eq:TotientProps}(a), the amount of steep rich slopes is
\begin{align*} 
\sum_{s=n^{2/3-\alpha}/k_s}^{n^{2/3-\alpha}\cdot k_s} &\phi_{s/n^{1-2\alpha}}(s) = \sum_{s=n^{2/3-\alpha}/k_s}^{n^{2/3-\alpha}\cdot k_s} \left(\frac{\phi(s)}{n^{1-2\alpha}}+O(2^{w(s)})\right) \\[2mm]
&= \frac{1}{n^{1-2\alpha}}\cdot \left(\sum_{s=1}^{n^{2/3-\alpha}\cdot k_s} \phi(s)-\sum_{s=1}^{n^{2/3-\alpha}/ k_s} \phi(s)\right) +O\left(\sum_{s=1}^{n^{2/3-\alpha}\cdot k_s} 2^{w(s)}-\sum_{s=1}^{n^{2/3-\alpha}/ k_s} 2^{w(s)}\right)\\[2mm]
&=\Theta\left(\frac{n^{4/3-2\alpha}}{n^{1-2\alpha}}(k_s^2-1/k_s^2)\right)+O\left(n^{2/3-\alpha}(k_s^2-1/k_s^2)\log\log n\right)= \Theta(n^{1/3}).
\end{align*}

We conclude that there are $\Theta(n^{1/3})$ families of $\Theta(n^{2/3})$ proper parallel lines with steep slopes.
These form $\Theta(n^{4/3})$ incidences. 

\parag{The $y$-intercepts of a parallel family.}
We consider a rich slope $s/t$.
Each proper line with slope $s/t$ is incident to a point from the $t$ leftmost columns or from the $s$ bottom rows of $[n^\alpha]\times [n^{1-\alpha}]$. 
Such a line that is incident to the point $(p_x,p_y)$ is defined by $y=\frac{s}{t}\cdot x+p_y-p_x\cdot \frac{s}{t}$. 
That is, the $y$-intercept is $p_y-p_x\cdot \frac{s}{t}$.
By the above, if $(p_x,p_y)$ is on one of the $t$ leftmost columns, then $p_x\in [t]$ and $p_y\in [n^{1-\alpha}-sn^{1/3}/k]$. 
If $(p_x,p_y)$ is on one of the $s$ bottom rows, then $p_x\in [n^\alpha-tn^{1/3}/k]$ and $p_y\in [s]$. 

Simplifying the preceding paragraph, the set of $y$-intercepts is
\[ \left\{j-i\cdot \frac{s}{t}\ :\  i\in [t],\ j\in [n^{1-\alpha}-sn^{1/3}/k],\ \text{ or }\ i\in [n^\alpha-tn^{1/3}/k],\ j\in [s]\right\}.\]
Assuming that $k$ is sufficiently large, the total number of possible values for $(p_x,p_y)$ is 
\[ \Theta(tn^{1-\alpha} + sn^{\alpha}) = \Theta(n^{2/3}). \]
The calculation depends on whether we are in the steep case or in the non-steep case.
Both cases repeat calculations that were already done above. 

Since there are $\Theta(n^{2/3})$ possible $y$-intercepts and the slope is rich, the $y$-intercepts of the slope must be a constant portion of all possible options.
\end{proof}

\section{Half-lattices, generalized half-lattices, and generalized lattices} \label{sec:generalizations}

In this section, we prove our structural results for lattice generalizations. 
Most of the technical work for these proofs is in Section \ref{sec:MultEnergy}, and here we add the final details.
We begin with the case of half-lattices. 
Recall that a Cartesian product $A \times B$ is a \emph{half-lattice} if at least one of $A$ and $B$ is an arithmetic progression. 
\vspace{2mm}

\noindent {\bf Theorem \ref{th:HalfLatticeConcurrentOnLattice}.}
\emph{Consider $\alpha>1/3$ and $B\subset \RR$ such that $|B|=n^{1-\alpha}$. 
Then the concurrent case of Theorem \ref{th:StructuralSzemTrot}(a) cannot occur with the half-lattice $[n^\alpha] \times B$. }
\begin{proof}
Consider a concurrent family of $\Theta(n^{\gamma})$ proper lines with center $p\in \RR^2$.
We denote this set of concurrent lines as $\lines$ and translate $\RR^2$ so that $p$ becomes the origin.
Let $A$ be the set of $x$-coordinates of the half-lattice after the translation. 
Abusing notation, we denote as $B$ and $\lines$ the sets obtained after the translation.

We remove axis-parallel lines from $\lines$.
Since $\lines$ contains at most two such lines, this does not affect the asymptotic size of $\lines$.
If $0\in B$ then we remove $0$ from $B$.
Since there are no axis-parallel lines in $\lines$, this does not remove any incidences from $(A\times B)\times \lines$, except possibly at the origin.

Fix a line $\ell\in \lines$.
Since $\ell$ is not axis-parallel, there exists $s\in \RR\setminus\{0\}$ such that $\ell$ is defined as $y=s\cdot x$. 
Given two points $(a,b),(a',b')\in A\times B$ that are on $\ell$, we have that $b=s\cdot a$ and $b' = s\cdot a'$.
Combining these equations leads to $a\cdot b' = a' \cdot b$.
That is, the quadruple $(a,b',a',b)\in A \times B \times A \times B$ contributes to $E^\times(A,B)$. 

Every point of $(A\times B)\setminus \{p\}$ is incident to at most one line of $\lines$.
Thus, different lines of $\lines$ cannot lead to the same quadruple that contributes to $E^\times(A,B)$. 
This leads to 
\begin{equation} \label{eq:MultEneryLower} 
E^\times(A,B) =\Omega(|\lines|\cdot n^{2/3}) = \Omega(n^{\gamma+2/3}).
\end{equation}

Theorem \ref{th:multEnergy} states that $E^{\times}(A, B) =O(n^{1 +\eps}+n^{2-2\alpha})$.
The relevant part of the theorem depends on the $x$-coordinate of $p$ before the translation. 
Combining this with \eqref{eq:MultEneryLower} implies that $\gamma\le \max\{1/3+\eps, 4/3-2\alpha\}$.
Since $\alpha>1/3$, we obtain that $\gamma<1-\alpha$.
That is, a family of concurrent lines with point of concurrency $p$ is of size $o(n^{1-\alpha})$.
The concurrent case of Theorem \ref{th:StructuralSzemTrot}(a) considers concurrent families of size $\Omega(n^{1-\alpha})$ and is thus impossible. 
To complete the proof, we note that the above holds for every $p\in \RR$.
\end{proof}

We next prove our results for generalized lattices and generalized half-lattices.
The proofs rely on the same ideas as the proof of Theorem \ref{th:HalfLatticeConcurrentOnLattice}.
However, since we do not have a variant of Theorem \ref{th:multEnergy} for this case, we obtain weaker bounds.
We recall that a Cartesian product $A\times B$ is a \emph{generalized lattice} if both $A$ and $B$ are constant-dimension generalized arithmetic progressions. 
\vspace{2mm}

\noindent {\bf Theorem \ref{th:ConcurrentSmallSumSet}.}
\emph{For $1/3< \alpha < 1/2$, let $A,B\subset \RR$ satisfy $|A|=n^{\alpha}$ and $|B|=n^{1-\alpha}$. \\[2mm]
(a) If $A\times B$ is a generalized lattice, then every concurrent family is of size $O(n^{1/3}\log n)$. \\[2mm]
(b) If $|B+B|=O(|B|)$ then every concurrent family is of size $O(n^{1/3+\alpha/2}\log^{1/2} n)$. \\[2mm]
(c) If $|A+A|=O(|A|)$ then every concurrent family is of size $O(n^{5/6-\alpha/2}\log^{1/2} n)$.}
\begin{proof}
(a) Assume for contradiction that there exists a family $\lines$ of $k\cdot n^{1/3}\log n$ concurrent proper lines. 
We translate $\RR^2$ so that the point of concurrency of $\lines$ become the origin.
A translation does not affect incidences, $|A+A|$, or $|B+B|$.

Let $\ell\in \lines$ be a non-axis-parallel line.
There exists $s\in \RR\setminus\{0\}$ such that $\ell$ is defined as $y=s\cdot x$. 
Given two points $(a,b),(a',b')\in (A\times B)\setminus\{(0,0)\}$ that are on $\ell$, we have that $b=s\cdot a$ and $b' = s\cdot a'$.
Combining these equations leads to $a\cdot b' = a' \cdot b$.
That is, the quadruple $(a,b',a',b)\in A \times B \times A \times B$ contributes to $E^\times(A,B)$. 
Excluding the origin, every pair of points on $\ell$ yield a different quadruple. 
Since $\ell$ is proper, it leads to $\Theta(n^{2/3})$ such quadruples. 

Every point of $A\times B\setminus\{0\}$ is incident to at most one line of $\lines$.
Thus, different lines of $\lines$ cannot lead to the same quadruple that contributes to $E^\times(A,B)$. 
This leads to 
\begin{equation} \label{eq:MultEneryLower2} 
E^\times(A,B) =\Omega(|\lines|\cdot n^{2/3}) = \Omega(k n\log n).
\end{equation}

For $q\in \RR$ and a finite $C\subset \RR$, we define 
\begin{equation*} 
r_C(q) = |\{(c,c')\in C^2\ :\ c/c'=q\}|.
\end{equation*}
If $0\in A$ or $0\in B$, then we remove 0 from those sets. 
These removals decrease $E^\times(A,B)$ by $O(n^{2\alpha}+n^{2-2\alpha})$. 
When $1/3< \alpha < 2/3$, this decrease is asymptotically smaller than the the lower bound of \eqref{eq:MultEneryLower2}, so it is negligible.
A quadruple $(a,b',a',b)\in A \times B \times A \times B$ that contributes to $E^\times(A,B)$ satisfies $a\cdot b' = a' \cdot b$.
Rearranging leads to $a/a'=b/b'$.
This implies that 
\begin{equation} \label{eq:AltMultEnergy}
E^{\times}(A,B)=\sum_{q\in \RR} r_A(q)\cdot r_B(q).
\end{equation}

Lemma \ref{le:Solymosi} implies that $E^\times(A) = O(|A+A|^2\log|A|) = O(n^{2\alpha}\log n)$ and that $E^\times(B) = O(n^{2-2\alpha}\log n)$.
By the Cauchy-Schwarz inequality and applying \eqref{eq:AltMultEnergy} multiple times, we obtain that 
\begin{align}
    E^{\times}(A,B) &= \sum_{q\in \RR} r_A(q) r_B(q) \le \left(\sum_{q\in \RR} r_A(q)^2\right)^{1/2}\cdot \left(\sum_{q\in \RR} r_B(q)^2\right)^{1/2} \nonumber \\[2mm]
    &= E^{\times}(A)^{1/2}\cdot E^{\times}(B)^{1/2}  =  O\left(n^{\alpha}\log^{1/2} n \cdot n^{1-\alpha}\log^{1/2} n\right) = O\left(n\log n\right). \label{eq:BipartiteEnergy(a)}
\end{align}

When $k$ is sufficiently large, the above contradicts \eqref{eq:MultEneryLower2}, so there is no family of $k\cdot n^{1/3}\log n$ concurrent proper lines.

(b) We imitate the proof of part (a).
In this case, \eqref{eq:MultEneryLower2} is replaced with 
\begin{equation*}  
E^\times(A,B) =\Omega(|\lines|\cdot n^{2/3}) = \Omega(kn^{1+\alpha/2}\log^{1/2} n).
\end{equation*}

By definition, we have that $E^\times(A)\le |A|^3 = n^{3\alpha}$.
Adapting \eqref{eq:BipartiteEnergy(a)} leads to 
\begin{align*}
    E^{\times}(A,B) = O\left(n^{3\alpha/2}\cdot n^{1-\alpha}\log^{1/2} n\right) = O\left(n^{1+\alpha/2}\log^{1/2}n\right).
\end{align*}
We end up with the same contradiction as in part (a).

(c) We again imitate the proof of part (a).
In this case, \eqref{eq:MultEneryLower2} is replaced with 
\begin{equation*}  
E^\times(A,B) =\Omega(|\lines|\cdot n^{2/3}) = \Omega(kn^{(3-\alpha)/2}\log^{1/2} n).
\end{equation*}

By definition, we have that $E^\times(B)\le |B|^3 = n^{3-3\alpha}$.
Adapting \eqref{eq:BipartiteEnergy(a)} leads to 
\begin{align*}
    E^{\times}(A,B) = O\left(n^{\alpha}\log^{1/2}n \cdot n^{(3-3\alpha)/2}\right) = O\left(n^{(3-\alpha)/2}\log^{1/2}n\right).
\end{align*}
We end up with the same contradiction as in part (a).
\end{proof}

\end{document}